\documentclass [a4paper, 11pt]{article}

\usepackage{amsmath}
\usepackage{amsfonts}
\usepackage{amssymb}
\usepackage{amsthm}
\usepackage{makeidx}
\usepackage{graphicx}
\usepackage{pb-diagram}
\usepackage{epstopdf}
\usepackage{fancyhdr}
\usepackage[all]{xy}

\newtheorem{cor}{Corollary}[section]
\newtheorem{te}[cor]{Theorem}
\newtheorem{p}[cor]{Proposition}
\newtheorem{lemma}[cor]{Lemma}
\theoremstyle{definition}
\newtheorem{de}[cor]{Definition}
\theoremstyle{remark}
\newtheorem{ob}[cor]{Observation}
\newtheorem{ex}[cor]{Exemple}

\newtheorem{nt}[cor]{Notation}

\newcommand{\cz}{\mathbb{C}}
\newcommand{\nz}{\mathbb{N}}
\newcommand{\zz}{\mathbb{Z}}

\newcommand{\ff}{\mathbb{F}}

\newcommand{\bb}{\mathcal{B}}
\newcommand{\gr}{\mathcal{G}}

\newcommand{\nr}{\mathcal{N}}
\newcommand{\unit}{\mathcal{U}}

\newcommand{\vp}{\varepsilon}

   {\begin{verse}%
     \small }%
   {\end{verse}}

\begin{document}

\setlength{\parskip}{1ex plus 0.5ex minus 0.2ex}
\begin{center}
$\mathbf{ON\,\,\,SOFIC\,\,\,ACTIONS\,\,\,AND\,\,\,EQUIVALENCE\,\,\,RELATIONS}$
\end{center}

\begin{center}
L. P\u AUNESCU\footnote{Work supported by the Marie Curie Research
Training Network MRTN-CT-2006-031962 EU-NCG.}
\end{center}

\begin{center}
\end{center}

$\mathbf{Abstract.}$ The notion of sofic equivalence relation was
introduced by Gabor Elek and Gabor Lippner. Their technics employ
some graph theory. Here we define this notion in a more operator
algebraic context, starting from Connes' embedding problem, and
prove the equivalence of this two definitions. We introduce a
notion of sofic action for an arbitrary group and prove that
amalgamated product of sofic actions over amenable groups is again
sofic. We also prove that amalgamated product of sofic groups over
an amenable subgroup is again sofic.

\tableofcontents
\section{Introduction}

\subsection{Definitions of hyperlinear and sofic actions}

We shall work with ultraproducts of matrix algebras instead of
$R^\omega$. We denote such a ultraproduct by
$\Pi_{k\to\omega}M_{n_k}(\cz)$. It is well known that embedding in
$R^\omega$ is equivalent with embedding in
$\Pi_{k\to\omega}M_{n_k}(\cz)$. For the sake of completeness we
include a proof. We shall always denote by $Tr$ the normalized
trace on a type $I$ or type $II$ finite factor.

\begin{p}
$R^\omega$ embeds in some $\Pi_{k\to\omega}M_{n_k}(\cz)$. Thus any
type II factor embedable in $R^\omega$ also embeds in a
ultraproduct of matrices.
\end{p}
\begin{proof}
Approximating the hyperfinite factor by matrix algebras we can
easily see that $R\subset M_{n_k}^\omega$. Then $R^\omega\subset
(M_{n_k}^\omega)^\omega\simeq M_{n_k}^{\omega\otimes\omega}$.
\end{proof}

Some results in this article will have to do with soficity of
certain objects. Because of this, we shall work with permutations.

\begin{nt}
We shall denote by $D_{n_k}\subset M_{n_k}$ the diagonal
subalgebra and by $P_{n_k}\subset M_{n_k}$ the subgroup of
permutation matrices. Given an ultraproduct
$\Pi_{k\to\omega}M_{n_k}(\cz)$ denote by
$\Pi_{k\to\omega}D_{n_k}(\cz)$ and $\Pi_{k\to\omega}P_{n_k}(\cz)$
the corresponding subsets.
\end{nt}

By a theorem of Popa (see \cite{Po1}, proposition 4.3)
$\Pi_{k\to\omega}D_{n_k}(\cz)$ is a maximal abelian nonseparable
subalgebra of $\Pi_{k\to\omega}M_{n_k}(\cz)$. We now introduce a
notion of hyperlinearity and soficity for actions.

\begin{de}
An action $\alpha$ of a countable group $G$ on a standard Borelian
space $(X,\bb,\mu)$ is called \emph{hyperlinear} if the crossed
product $L^\infty(X)\rtimes_\alpha G$ embeds in $R^\omega$.
\end{de}

\begin{de}\label{sofic action}
An action $\alpha$ of a countable group $G$ on a standard Borelian
space $(X,\bb,\mu)$ is called \emph{sofic} if the crossed product
$L^\infty(X)\rtimes_\alpha G$ embeds in
$\Pi_{k\to\omega}M_{n_k}(\cz)$ such that
$L^\infty(X)\subset\Pi_{k\to\omega}D_{n_k}(\cz)$ and
$G\subset\Pi_{k\to\omega}P_{n_k}(\cz)$.
\end{de}

So for an action to be sofic instead of simply hyperlinear we want
that the unitaries implementing the action, to be permutation
matrices. Note that also Elek and Lippner defined a notion of
soficity for action of $\ff_\infty$ as a preliminary step in
defining the concept of sofic equivalence relation. The two
notions of sofic actions are not equivalent.

Next theorem shows that the property of sofic action is invariant
under orbit equivalence. While the proof is quite simple, this
theorem hints very clearly that being sofic is a property of the
orbit equivalence relation, rather than of the action itself.

\begin{te}\label{orbit equivalent actions}
Let $\alpha$ and $\beta$ be two free orbit equivalent actions. If
$\alpha$ is hyperlinear (sofic) then also $\beta$ is hyperlinear
(sofic).
\end{te}
\begin{proof}
Let $\alpha$ be an action of $G$ and $\beta$ be an action of $H$
such that $L^\infty(X)\rtimes_\alpha G\simeq
L^\infty(X)\rtimes_\beta H$ (isomorphism that is identity on
$L^\infty(X)$). By this we are done with the hyperlinear part of
the theorem. Consider now a sofic embedding of
$L^\infty(X)\rtimes_\alpha G$ in $\Pi_{k\to\omega}M_{n_k}$. Let
$u_g\in L(G)$, $g\in G$ and $v_s\in L(H)$, $s\in H$ denoting the
corresponding unitaries. For $s\in H$ we can find projections
$\{p_g|g\in G\}$ in $L^\infty(X)$ such that $\sum_gp_g=1$ and
$v_s=\sum_gp_gu_g$. By next lemma, $v_s\in
\Pi_{k\to\omega}P_{n_k}$ and we are done.
\end{proof}

\begin{lemma}\label{permutations}
Let $\{e_i|i\in\nz\}$ be projections in $\Pi_{k\to\omega}D_{n_k}$
such that $\sum_ie_i=1$. Let $\{u_i|i\in\nz\}$ be unitary elements
in $\Pi_{k\to\omega}P_{n_k}$ such that $\sum_iu_i^*e_iu_i=1$. Then
$v=\sum_ie_iu_i$ is a unitary in $\Pi_{k\to\omega}P_{n_k}$.
\end{lemma}
\begin{proof}
Using $\sum_ie_i=1$ we can easily construct projections $e_i^k\in
D_{n_k}$ such that:
\begin{enumerate}
\item $e_i=\Pi_{k\to\omega}e_i^k$; \item $\sum_ie_i^k=1_{n_k}$.
\end{enumerate}
By hypothesis we have $u_i=\Pi_{k\to\omega}u_i^k$ where $u_i^k\in
P_{n_k}$. If $v^k=\sum_ie_i^ku_i^k$ then $v=\Pi_{k\to\omega}v^k$,
but $v^k$ are not necessary permutation matrices. Because
$\sum_iu_i^*e_iu_i=1$ we have $v^*v=1$, so $v$ is a unitary. Now,
$v^k$ has only $0$ and $1$ entries and exactly one entry of $1$ on
each line. Then $v^{k*}v^k$ is a diagonal matrix giving the number
of $1$ entries on each column.

Denote by $r_k$ the number of columns in $v^k$ having only $0$
entries. This number will also represent the numbers of $0$
entries on the diagonal of $v^{k*}v^k$. Then:
\[||v^{k*}v^k-Id||_2^2\geq\frac{r_k}{n_k}.\]
Because $\Pi_{k\to\omega}v^{k*}v^k=1$ we have
$r_k/n_k\to_{k\to\omega}0$. We can "move" $r_k$ entries of $1$ in
$v^k$ on the $r_k$ missing columns to get a permutation matrix
$w^k\in P_{n_k}$. Then:
\[||v^k-w^k||_2^2=\frac{2r_k}{n_k}.\]
Combined with $r_k/n_k\to_{k\to\omega}0$ we get
$v=\Pi_{k\to\omega}w^k$. This will prove the lemma.
\end{proof}

\begin{de}
We shall call an element of $\Pi_{k\to\omega}P_{n_k}$ a
\emph{permutation}, instead of an ultraproduct of permutations. A
\emph{piece of permutation} will be an element of the form $e\cdot
p$, where $e$ is a projection in $\Pi_{k\to\omega}D_{n_k}$ and $p$
is a permutation.
\end{de}

Lemma \ref{permutations} now becomes \emph{a unitary that is a sum
of pieces of permutations is a permutation}.

\subsection{The Feldman-Moore construction}

Based on theorem \ref{orbit equivalent actions}, a definition of
sofic equivalence relation is possible for relations generated by
free actions. Unfortunately, not all equivalence relations have
this property. We adapt the definition using the Feldman-Moore
construction. Let us recall some things from \cite{Fe-Mo}.

Let $(X,\bb,\mu)$ be a standard space with a probability measure.
Let $E\subset X^2$ an equivalence relation on $X$ such that
$E\in\bb\times\bb$. We shall work only with equivalence relations
that are countable, i.e. every equivalence class is countable, and
$\mu$-invariant. Before we recall what this means we introduce
some notation.

Denote by $[E]$ set of all isomorphism with graph in $E$ and by
$[[E]]$ set of all partial isomorphism with graph in $E$:
\begin{align*}
[E]=&\{\theta:X\to X:\theta\mbox{ bijection },graph\theta\subset E\};\\
[[E]]=&\{\phi:A\to B:A,B\subset X,\ \phi\mbox{ bijection
},graph\phi\subset E\}.
\end{align*}

\begin{de}
Let $E$ an equivalence relation on $(X,\mu)$. Then $E$ is called
$\mu$-invariant if for any $\phi:A\to B$, $\phi\in[[E]]$ we have
$\mu(A)=\mu(B)$.
\end{de}

For a von Neumann algebra $A$, we shall denote by $\unit(A)$ the
group of unitaries in the algebra $A$. The normalizer $\nr_M(A)$
and the normalizing pseudogroup $\gr\nr_M(A)$ are the
corresponding objects of $[E]$ and $[[E]]$ respectively. For
$A\subset M$ define:

\begin{align*}
&\nr_M(A)=\{u\in\unit(M):uAu^*=A\};\\
&\gr\nr_M(A)=\{v\in M\mbox{ partial isometry}:vv^*,v^*v\in A,
vAv^* =Avv^*\}.
\end{align*}

Now we can construct the algebra $M(E)$ associated to an
equivalence relation.

\begin{de}
A measurable function $a:E\to\cz$ is called \emph{finite} if $a$
is bounded and there is a natural number $n$ such that:
\begin{align*}
Card(\{x:a(x,y)\neq0\})\leq n\ \forall y\in X;\\
Card(\{y:a(x,y)\neq0\})\leq n\ \forall x\in X.
\end{align*}
\end{de}

A finite function (matrix) is a bounded function with finite
number of nonzero entries on each line and column (having also a
global margin). We shall multiply this functions as general
matrices and the definition of finite function guarantees we get a
$*$-algebra. Define:
\begin{align*}
&M_0(E)=\{a:E\to\cz:a\mbox{ finite function}\};\\
&a\cdot b(x,z)=\sum_ya(x,y)b(y,z);\\
&a^*(x,y)=\overline{a(y,x)}.
\end{align*}

It is easy to check this is indeed a $*$-algebra. The trace is
defined in a similar way as in the case of matrices:
\[Tr(a)=\int_Xa(x,x)d\mu.\]

The algebra $M(E)$ will be the weak closure of $M_0(E)$ in the GNS
representation of $(M_0(E),Tr)$. By general theory of von Neumann
algebras, using the cyclic separating vector of the GNS
representation, we can still see elements of $M(E)$ as measurable
functions $a:E\to\cz$.

Let $\Delta=\{(x,x):x\in X\}$ the diagonal in $E$. Define the
subalgebra of diagonal matrices:
\[A=\{a\in M(E):supp(a)\subset\Delta\}.\]

We shall denote by $\delta_x^y$ the Kronecker delta function, i.e.
$\delta_x^y=1$ iff $x=y$; otherwise $\delta_x^y=0$. Notation
$\chi_A$ stands for the characteristic function of $A$.

\begin{de}
For $\theta\in[E]$ define $u_\theta\in M(E)$ by:
$u_\theta(x,y)=\delta_x^{\theta(y)}$. For $\phi\in[[E]]$, define
$v_\phi(x,y)=\chi_{dom(\phi)}(y)\cdot\delta_x^{\phi(y)}$.
\end{de}

It is not hard to see that $u_\theta\in\nr(A)$ for any
$\theta\in[E]$ and $v_\phi\in\gr\nr(A)$ for any $\phi\in[[E]]$.
Moreover any $u\in\nr(A)$ is of the form $a\cdot u_\theta$, where
$a\in\unit(A)$ and $\theta\in[E]$. Also $u_\theta
u_\psi=u_{\theta\circ\psi}$ for $\theta,\psi\in[E]$. This provides
a group isomorphism between the Weyl group $\nr(A)/\unit(A)$ and
$[E]$.

Algebra $A$ is maximal abelian in $M(E)$. Also $\nr(A)''=M(E)$.
This properties make $A$ a Cartan subalgebra of $M(E)$. We shall
call $A\subset M(E)$ a Cartan pair.

Motivation of Feldman-Moore construction was the invariance of
crossed product up to orbit equivalent actions. Next example shows
this is indeed the right construction.

\begin{ex}
Let $\alpha:G\to Aut(X,\mu)$ a free action. Denote by $E_\alpha$
the orbit equivalence relation generated by $\alpha$ on $X$. Then:
\[L^\infty(X)\rtimes_\alpha G\simeq M(E_\alpha).\]
\end{ex}

\subsection{Definition of sofic equivalence relations}

The notion of sofic equivalence relation was introduced by Gabor
Elek and Gabor Lippner (see \cite{El-Li}). We shall provide a
different definition here and prove in section \ref{Ser} the
equivalence of the two definitions.

\begin{de}
An equivalence relation $E$ is called \emph{sofic} if there is an
embedding of $M(E)$ in some $\Pi_{k\to\omega}M_{n_k}$ such that
$A\subset\Pi_{k\to\omega}D_{n_k}$ and $\nr(A)\subset
\unit(A)\cdot\Pi_{k\to\omega}P_{n_k}$.
\end{de}

This has the advantage of being a compact definition, but in
practice we shall need the following type of embeddings.

\begin{de}
Let $E$ an equivalence relation and $A\subset M(E)$ the Cartan
pair associated to $E$. We call an embedding
$\Theta:M(E)\to\Pi_{k\to\omega}M_{n_k}$ \emph{sofic} if
$\Theta(A)\subset\Pi_{k\to\omega}D_{n_k}$ and
$\Theta(u_\theta)\subset\Pi_{k\to\omega}P_{n_k}$ for any
$\theta\in[E]$.
\end{de}

\begin{p}\label{soficembedding}
An equivalence relation $E$ is a sofic if and only if its Cartan
pair $A\subset M(E)$ admits a sofic embedding.
\end{p}
\begin{proof}
Let $\Theta:M(E)\to\Pi_{k\to\omega}M_{n_k}$ an embedding such that
$\Theta(A)\subset\Pi_{k\to\omega}D_{n_k}$ and
$\Theta(\nr(A))\subset\Theta(\unit(A))\cdot\Pi_{k\to\omega}P_{n_k}$.

For $\varphi\in[E]$ we have a unique decomposition
$\Theta(u_\varphi)=\Theta(f_\varphi)v_\varphi$, where
$f_\varphi\in\unit(A)$ and $v_\varphi\in\Pi_{k\to\omega}P_{n_k}$.
Then:
\[\Theta(f_\psi\circ\varphi^{-1})=\Theta(u_\varphi)\Theta(f_\psi)\Theta(u_\varphi^*)=
\Theta(f_\varphi)(v_\varphi\Theta(f_\psi)v_\varphi^*)\Theta(f_\varphi^*)=v_\varphi\Theta(f_\psi)v_\varphi^*.\]
Because of the uniqueness of the decomposition of
$\Theta(u_\varphi)$ we have
$f_{\varphi\psi}=f_\varphi(f_\psi\circ\varphi^{-1})$. If
$\chi_\varphi$ denotes the projection with support $\{x\in
X:\varphi(x)=x\}$, one has $\chi_\varphi u_\varphi=\chi_\varphi$
and hence:
\[\Theta(f_\varphi^*\chi_\varphi)=\Theta(f_\varphi^*\chi_\varphi
u_\varphi)=\Theta(\chi_\varphi)v_\varphi.\]

The conditional expectation of $v_\varphi$ on
$\Pi_{k\to\omega}D_{n_k}$ is a projection. Thus, taking the
conditional expectation on $\Theta(A)$ it follows that
$f_\varphi^*\chi_\varphi$ is positive and hence equal to
$\chi_\varphi$. So, for all $\varphi\in[E]$ we have
$f_\varphi\chi_\varphi=\chi_\varphi$. Altogether, it follows that
the formula:
\begin{align*}
\alpha(u_\varphi)&=f_\varphi^*u_\varphi\mbox{ for all
}\varphi\in[E],\\
\alpha(a)&=a\mbox{ for all }a\in A
\end{align*}
provides a well defined automorphism of $M(E)$. The composition of
$\Theta$ and $\alpha$ is the required sofic embedding of $M(E)$.
\end{proof}

As a consequence of this proposition and lemma \ref{permutations},
we have the following result.

\begin{p}\label{equivalent definitions}
Let $\alpha$ be a free action. Then $E_\alpha$ is a sofic
equivalence relation if and only if $\alpha$ is a sofic action.
\end{p}

\subsection{Preliminaries}

We include here some propositions that will be used in different
situations. First an easy observation.

\begin{ob}\label{differentdimensions}
Let $\Theta=\Pi_{k\to\omega}\Theta_k$ be a sofic embedding of some
von Neumann algebra $M$ in $\Pi_{k\to\omega}M_{n_k}$. Consider
also $\{r_k\}_k$ a sequence of natural numbers. Then
$\Theta\otimes 1= \Pi_{k\to\omega}\Theta_k\otimes 1_{r_k}$ is
again a sofic embedding of $M$ in $\Pi_{k\to\omega}M_{n_k}\otimes
M_{r_k}=\Pi_{k\to\omega}M_{n_kr_k}$.

This trick will be used when we need to embed two algebras in the
same $\Pi_{k\to\omega}M_{n_k}$ (that is the same matrix dimension
at each step).
\end{ob}

Next sequence of lemmas will show that two sofic embeddings of the
same hyperfinite algebra are conjugate by a permutation.

\begin{lemma}
Let $e,f$ two projections in $\Pi_{k\to\omega}D_{n_k}$ such that
$Tr(e)=Tr(f)$. Then there is a unitary
$u\in\Pi_{k\to\omega}P_{n_k}$ such that $f=ueu^*$.
\end{lemma}
\begin{proof}
Let $e=\Pi_{k\to\omega}e^k$ and $f=\Pi_{k\to\omega}f^k$ such that
$e^k$ and $f^k$ are projections in $D_{n_k}$. Assume $e^k$ has
$t_k$ entries of $1$ and $f^k$ has $s_k$ entries of $1$, so
$\lim_{k\to\omega}t_k/n_k=Tr(e)=Tr(f)= \lim_{k\to\omega}s_k/n_k$.
Choose $p_1^k\in P_{n_k}$ such that $p_1^ke^kp_1^{k*}$ has the
first $t_k$ entries of $1$ on the diagonal. In the same way choose
$p_2^k$ such that $p_2^kf^kp_2^{k*}$ has the first $s_k$ entries
of $1$ on the diagonal. Define $p_i=\Pi_{k\to\omega}p_i^k$ for
$i=1,2$. Our constructions guarantee that
$Tr(|p_1ep_1^*-p_2fp_2^*|)= \lim_{k\to\omega}|t_k-s_k|/n_k=0$.
Then $p_1ep_1^*=p_2fp_2^*$ so define $u=p_2^*p_1$.
\end{proof}

\begin{lemma}
Let $\{e_i\}_{i=1}^m$ and $\{f_i\}_{i=1}^m$ two sequences of
projections in $\Pi_{k\to\omega}D_{n_k}$ such that
$\sum_{i=1}^me_i=1=\sum_{i=1}^mf_i$ and $Tr(e_i)=Tr(f_i)$ for each
$i=1,\ldots,m$. Then there is a unitary
$u\in\Pi_{k\to\omega}P_{n_k}$ such that $f_i=ue_iu^*$ for all
$i=1,\ldots,m$.
\end{lemma}
\begin{proof}
Apply previous lemma for each $i=1,\ldots,m$ to get elements
$u_i\in\Pi_{k\to\omega}P_{n_k}$ such that $u_ie_iu_i^*=f_i$.
Define $u=\sum_{i=1}^mu_ie_i$. Then by lemma (\ref{permutations})
we know $u\in\Pi_{k\to\omega}P_{n_k}$. Also
$ue_iu^*=u_ie_iu_i^*=f_i$.
\end{proof}

\begin{p}\label{linkabelian}
Let $\Theta_1,\Theta_2$ be two embeddings of $L^\infty(X)$ in
$\Pi_{k\to\omega}D_{n_k}$. Then there exists a unitary $u\in
\Pi_{k\to\omega}P_{n_k}$ such that $\Theta_2(a)=u\Theta_1(a)u^*$
for every $a\in L^\infty(X)$.
\end{p}
\begin{proof}
Let $A_m$ an increasing sequence of commutative finite dimensional
subalgebras such that $L^\infty(X)=(\cup_mA_m)''$. By previous
lemma there exist a unitary $u_m\in\Pi_{k\to\omega}P_{n_k}$ such
that $\Theta_2(a)=Ad u_m\circ\Theta_1(a)$ for $a\in A_m$. We shall
construct $u\in \Pi_{k\to\omega}P_{n_k}$ using a diagonal
argument. Let $u_m=\Pi_{k\to\omega}u_m^k$ with $u_m^k\in P_{n_k}$
and $\Theta_i(a)=\Pi_{k\to\omega}\Theta_i(a)_k$ with
$\Theta_i(a)_k\in D_{n_k}$.

Inductively choose smaller $F_m\in\omega$, $m\in\nz$ such that
$||u_m^k\Theta_1(a)_ku_m^{k*}-\Theta_2(a)_k||_2<1/m$ for any $a\in
(A_m)_1$, $k\in F_m$. Define $u^k=u_m^k$ for $k\in F_m\setminus
F_{m+1}$ and set $u=\Pi_{k\to\omega}u^k$.
\end{proof}

\begin{p}\label{linkamenable}
Let $E$ be a hyperfinite equivalence relation and $A\subset M(E)$
the Cartan pair associated to $E$. Let $\Theta_1,\Theta_2$ two
sofic embeddings of $M(E)$ in $\Pi_{k\to\omega}M_{n_k}$. Then
there exists a unitary $u\in\Pi_{k\to\omega}P_{n_k}$ such that
$\Theta_2(x)=u\Theta_1(x)u^*$ for every $x\in M(E)$.
\end{p}
\begin{proof}
Using the previous proposition we can assume $\Theta_1$ and
$\Theta_2$ coincide on $A$. We shall first prove this result in
case of ergodic equivalence relation, i.e. $M(E)$ is the
hyperfinite factor. By definition of hyperfinite equivalence
relation and Feldman-Moore construction (see also proof of 4.1
from \cite{Po2}) there exists an increasing sequence of matrix
algebras $\{N_m\}_{m\geq1}$ of $M(E)$ each of them with a set of
matrix units $\{e_{ij}^m\}$ such that:
\begin{enumerate}
\item $M(E)$ is the weak closure of $\cup_m N_m$; \item
$e_{ii}^m\in A$ and $\sum_ie_{ii}^m=1$; \item $e_{ij}^m$ are of
the form $v_\theta$ with $\theta\in[[E]]$; \item every $e_{rs}^p$,
for $p\leq m$, is the sum of some $e_{ij}^m$.
\end{enumerate}
Elements $v_\theta$ are of the form $e\cdot u_\phi$, where $e$ is
a projection in $A$ and $\phi\in[E]$. Combined with $\Theta_l$ is
sofic, we get $\Theta_l(e_{ij}^m)$ is a piece of permutation.
Define
\[p_m=\sum_j\Theta_2(e_{j1}^m)\Theta_1(e_{1j}^m).\]
Then
\begin{align*}
p_mp_m^*=&\sum_{i,j}\Theta_2(e_{i1}^m)\Theta_1(e_{1i}^m)
\Theta_1(e_{j1}^m)\Theta_2(e_{1j}^m)\\
=&\sum_j\Theta_2(e_{j1}^m)\Theta_1(e_{11}^m)\Theta_2(e_{1j}^m)=
\sum_j\Theta_2(e_{jj}^m)=1,
\end{align*}
so $p_m$ is a unitary. Using \ref{permutations} we have
$p_m\in\Pi_{k\to\omega}P_{n_k}$. Moreover:
\begin{align*}
p_m\Theta_1(e_{rs}^m)p_m^*=&\sum_{i,j}\Theta_2(e_{i1}^m)\Theta_1(e_{1i}^m)\Theta_1(e_{rs}^m)
\Theta_1(e_{j1}^m)\Theta_2(e_{1j}^m)\\
=&\Theta_2(e_{r1}^m)\Theta_1(e_{11}^m)\Theta_2(e_{1s}^m)=
\Theta_2(e_{rs}^m).
\end{align*}
We obtained $p_m\Theta_1(x)p_m^*=\Theta_2(x)$ for $x\in N_m$.
Employing another diagonal argument we construct a permutation
$p\in \Pi_{k\to\omega}P_{n_k}$ such that
$p\Theta_1(x)p^*=\Theta_2(x)$ for $x\in \cup_mN_m$. Using 1 we are
done.

The proof in general case works the same. The only difference is
that $\{N_m\}_{m\geq1}$ are finite dimensional algebras instead of
matrix algebras, so we need to be more careful when defining
$p_m$. Assume that $N_m=N_m^1\oplus N_m^2\oplus\ldots\oplus
N_m^t$, with $N_m^v$ factors for $v=1,\ldots,t$. Let
$\{e_{ij;v}^m\}$ a set of matrix units for $N_m^v$. Then define:
\[p_m=\sum_{v=1}^t\sum_j\Theta_2(e_{j1;v}^m)\Theta_1(e_{1j;v}^m).\]
Computations that $p_m$ is a unitary and
$p_m\Theta_1(e_{rs}^m)p_m^*=\Theta_2(e_{rs}^m)$ are the same.
\end{proof}

\section{Bernoulli shifts}

In \cite{El-Li} Elek and Lippner proved that equivalence relations
generated by Bernoulli shifts of sofic groups are sofic. We
present here the nice proof of Narutaka Ozawa from \cite{Oz}.

\begin{te}(Elek-Lippner)
Equivalence relations generated by Bernoulli shifts of sofic
groups are sofic.
\end{te}
\begin{proof}(Ozawa)
Let $G$ be a sofic group. Every Bernoulli shift is a free action.
Using \ref{equivalent definitions} we just need to prove that each
Bernoulli shift of $G$ is a sofic action.

Let $X=\{0,1\}^G=\{f:G\to\{0,1\}\}$. For distinct $g_1,g_2\ldots
g_m$, define the cylinder set:
\[c_{g_1,g_2,\ldots,g_m}^{i_1,i_2,\ldots,i_m}=\{f\in
X:f(g_j)=i_j\ \ \forall j=1\ldots m\},\] and let
$Q_{g_1,g_2,\ldots,g_m}^{i_1,i_2,\ldots,i_m}$ be the projection
onto this set. Then $\beta$ is the action of $G$ on $X$ such that
$\beta(g)c_{g_1,g_2,\ldots,g_m}^
{i_1,i_2,\ldots,i_m}=c_{gg_1,gg_2,\ldots,gg_m}^
{i_1,i_2,\ldots,i_m}$.

Let $\Theta_0:G\to\Pi_{k\to\omega}P_{n_k}$ be a sofic embedding of
$G$ with $Tr(\Theta_0(g))=0$ for each $g\neq e$. Write
$\Theta_0(g)=\Pi_{k\to\omega}p_{g;k}$ such that $p_{g;k}\in
P_{n_k}$. Define $\Theta:G\to\Pi_{k\to\omega}M_{n_k}\otimes
M_{2^{n_k}}$ by $\Theta=\Theta_0\otimes 1$. Let $Y_k$ a set with
$n_k$ elements and identify $D_{n_k}$ with $L^\infty(Y_k)$. Also
let $Z_k=\{\eta:Y_k\to\{0,1\}\}$ and identify $D_{2^{n_k}}$ with
$L^\infty(Z_k)$. Define now:
\[c_{g_1,g_2,\ldots,g_m;k}^{i_1,i_2,\ldots,i_m}=\{(\xi,\eta)\in
Y_{n_k}\times Z_{n_k}:\eta(p_{g_j;k}^{-1}(\xi))=i_j,\
j=1,\ldots,m\}.\] Let
$Q_{g_1,g_2,\ldots,g_m;k}^{i_1,i_2,\ldots,i_m}\in D_{n_k}\otimes
D_{2^{n_k}}$ be the characteristic function of
$c_{g_1,g_2,\ldots,g_m;k}^{i_1,i_2,\ldots,i_m}$. Define now
$\Theta(Q_{g_1,g_2,\ldots,g_m}^
{i_1,i_2,\ldots,i_m})=\Pi_{k\to\omega}Q_{g_1,g_2,\ldots,g_m;k}^
{i_1,i_2,\ldots,i_m}$. Then:
\begin{align*}
\Theta(g)\Theta(Q_{g_1,g_2,\ldots,g_m}^ {i_1,i_2,\ldots,i_m})
\Theta(g)^*=&\Pi_{k\to\omega}(p_{g;k}\otimes 1)
Q_{g_1,g_2,\ldots,g_m;k}^{i_1,i_2,\ldots,i_m}(p_{g;k}^{-1}\otimes
1)\\=&\Pi_{k\to\omega}\chi_{\{(\xi,\eta):(p_{g;k}^{-1}\otimes
1)(\xi,\eta)\in c_{g_1,g_2,\ldots,g_m;k}^ {i_1,i_2,\ldots,i_m}\}}\\
=&\Pi_{k\to\omega}\chi_{\{(\xi,\eta):\eta(p_{g_j;k}^{-1}p_{g;k}^{-1}(\xi))=i_j,\
j=1,\ldots,m\}}\\=&_{not}\Pi_{k\to\omega}\chi_{T_k}.\\
\Theta(Q_{gg_1,gg_2,\ldots,gg_m}^
{i_1,i_2,\ldots,i_m})=&\Pi_{k\to\omega}\chi_{\{(\xi,\eta):\eta(p_{gg_j;k}^{-1}(\xi))=i_j,\
j=1,\ldots,m\}}\\ =&_{not}\Pi_{k\to\omega}\chi_{S_k}.
\end{align*}
If $(\xi,\eta)\in T_k\Delta S_k$ then for some $j=1,\ldots,m$ we
have $p_{g_j;k}^{-1}p_{g;k}^{-1}(\xi)\neq p_{gg_j;k}^{-1}(\xi)$.
Given the fact that $\Theta_0$ is a sofic embedding it follows
that $\Pi_{k\to\omega}\chi_{T_k}=\Pi_{k\to\omega}\chi_{S_k}$.

The only thing left is to compute the trace of
$\Theta(Q_{g_1,g_2,\ldots,g_m}^{i_1,i_2,\ldots,i_m})$. For this,
let $A_k=\{\xi\in Y_k:p_{g_j;k}^{-1}(\xi)\mbox{ are different for
}j=1,\ldots,m\}$. Because $Tr(\Theta_0(g))=0$ for $g\neq e$ we
have $\lim_{k\to\omega}Card(A_k)/n_k=1$. Then:
\[Tr(\Theta(Q_{g_1,g_2,\ldots,g_m}^ {i_1,i_2,\ldots,i_m}))=
\lim_{k\to\omega}Tr(Q_{g_1,g_2,\ldots,g_m;k}^
{i_1,i_2,\ldots,i_m})=\lim_{k\to\omega}\frac1{n_k2^{n_k}}(\sum_{\xi\in
A_k}2^{n_k-m}+\sum_{\xi\notin A_k}v_\xi)=\frac1{2^m}.\] This will
prove that $\Theta$ is an embedding of
$L^\infty(X)\rtimes_{\beta}G$, proving the soficity of the action
$\beta$.

The proof can be adapted to work for any Bernoulli shift. For a
finite uniform Bernoulli shift the proof works the same. A
diagonal argument will prove the theorem in case $X=[0,1]^G$ (with
product of Lebesgue measure). Any other Bernoulli shift will yield
a subalgebra of $L^\infty([0,1]^G)\rtimes G$.
\end{proof}

Next easy proposition will be used in the proof of corollary
\ref{amalgamatedgroupproduct}.

\begin{p}
Let $G$ act freely on a countable set $I$. Then the generalized
Bernoulli shift of $G$ on $\{0,1\}^I$ is sofic.
\end{p}
\begin{proof}
If $G$ acts freely on $I$ then $I$ is of the form $G\times I'$ and
the action is a shift on the first component. The generalized
Bernoulli shift on $\{0,1\}^I$ is a classic Bernoulli shift on
$X^G$ where $X=\{0,1\}^{I'}$.
\end{proof}

A formally weaker version of the following result was first obtain
by Benoit Collins and Ken Dykema (\cite{Co-Dy}). Independently,
Elek and Szabo proved this theorem using different methods
(\cite{El-Sz2}).

\begin{cor}\label{amalgamatedgroupproduct}
Amalgamated products of sofic groups over amenable groups are
sofic.
\end{cor}
\begin{proof}
Let $G_1$, $G_2$ be two sofic groups with a common amenable
subgroup $H$. Let $X=\{0,1\}^{G_1*_HG_2}$ equipped with product
measure. Then $G_1$ and $G_2$ act on $X$ as generalized Bernoulli
shifts and this actions coincide on $H$. Using the above
proposition (and \ref{differentdimensions}) we can construct sofic
embeddings $\Theta_i:L^\infty(X)\rtimes
G_i\to\Pi_{k\to\omega}M_{n_k}$ for $i=1,2$. By proposition
(\ref{linkamenable}) we can assume $\Theta_1=\Theta_2$ on
$L^\infty(X)\rtimes H$ (here we use $H$ amenable and classic
result from \cite{CFW}). Note that now $\Theta_1$ acts on
$\Theta_i(L^\infty(X))$ by shifting with $G_1$ and $\Theta_2$ acts
on the same space by shifting with $G_2$. This will provide a
representation $\Theta$ of $G_1*_HG_2$ on
$\Pi_{k\to\omega}P_{n_k}$.  Also, $\Theta$ acts on
$\Theta_i(L^\infty(X))$ as a classic Bernoulli shift. This implies
$\Theta$ is faithful, so $G_1*_HG_2$ is sofic.
\end{proof}

\begin{cor}
Let $H$ be an abelian group and $G$ a sofic group. Then $H\wr G$
(wreath product) is sofic.
\end{cor}
\begin{proof}
We shall work with the following presentation $<S|R>$ of the
wreath product:
\begin{align*}
S=&\{f_g^h,u_g:\mbox{ for every } h\in H\mbox{ and }g\in G\};\\
R=&\{f_g^e=e:\forall g\in G\}\cup
\{f_g^{h_1}f_g^{h_2}=f_g^{h_1h_2}:\forall g\in G,\forall
h_1,h_2\in H\}\cup\\ &\{f_{g_1}^{h_1}f_{g_2}^{h_2}=
f_{g_2}^{h_2}f_{g_1}^{h_1}:\forall g_1,g_2\in G,g_1\neq g_2\forall
h_1,h_2\in H\}\cup\\ &\{u_{g_1}u_{g_2}=u_{g_1g_2}:\forall
g_1,g_2\in G\}\cup\\
&\{u_{g_1}f_{g_2}^hu_{g_1}^{-1}=f_{g_1g_2}^h:\forall g_1,g_2\in G,
h\in H\}.
\end{align*}

Consider first the case of $\zz_2\wr G$. Apply Elek-Lippner result
to embed $L(\zz_2^G)\rtimes_\beta G\simeq L(\zz_2^G\rtimes
G)=L(\zz_2\wr G)$ in some $\Pi_{k\to\omega}M_{n_k}$. Generators
$u_g$ will be ultraproduct of permutations. Instead, elements of
the type $f_g^h$ are unitaries in $\Pi_{k\to\omega}D_{n_k}$ with
$\pm 1$ entries. Construct a sofic representation of $\zz_2\wr G$
in $\Pi_{k\to\omega}P_{2n_k}$ by replacing a $1$ entry with $I_2$
and a $-1$ entry with:$\left(%
\begin{array}{cc}
  0 & 1 \\
  1 & 0 \\
\end{array}%
\right)$.

Consider now the general case. Let
$\Theta:L^\infty(\{0,1\}^G)\rtimes G\to\Pi_{k\to\omega} M_{n_k}$
the sofic embedding constructed in the last proof. Also let
$\Lambda:H\to P_{m_k}$ be a sofic embedding of $H$. We shall
construct $\Phi:H\wr G\to\Pi_{k\to\omega}P_{n_k}\otimes P_{m_k}$
as follows:
\begin{align*}
&\Phi(u_g)=\Theta(g)\otimes 1;\\
&\Phi(f_g^h)=c_g^0\otimes 1+c_g^1\otimes\Lambda(h).
\end{align*}
Relations in set $R$ are easy to check (one needs $H$ abelian for
$f_{g_1}^{h_1}f_{g_2}^{h_2}= f_{g_2}^{h_2}f_{g_1}^{h_1}$). Also
$Tr(\Phi(u_g))=0$ and $Tr(\Phi(f_g^h))=1/2$. In order to finish
the proof we need to see that $\Phi$ is injective.

The genereic element of $H\wr G$ is $s=f_{g_1}^{h_1}f_{g_2}^{h_2}
\ldots f_{g_n}^{h_n}u_g$ with $g_1,g_2\ldots g_n$ distinct. Then:
\[\Phi(s)=\left(\sum_{(i_1,\ldots,i_n)\in\{0,1\}^n} c_{g_1,g_2\ldots g_n}^{i_1,i_2\ldots
i_n}\otimes\Lambda(\Pi_{i_k=1}h_k)\right)(\Theta(g)\otimes 1).\]
Assume $\Phi(s)=1$. Then for any $(i_1,\ldots,i_n)\in\{0,1\}^n$,
$\Lambda(\Pi_{i_k=1}h_k)=1$. This force $h_k=e$ for any $k$. Then
$\Phi(u_g)=1$, so $u_g=e$. It follows $s=e$.
\end{proof}

\section{Sofic actions}

The goal would be to prove that every (free) action of a sofic
group is sofic. While this remains open we shall prove this fact
for a family of groups. Let's first solve this ambiguity: free or
general actions.

\begin{te}\label{freesoficactions}
Let $G$ be a group such that every free action is sofic. Then
every action of $G$ is sofic.
\end{te}
\begin{proof}
Let $\alpha$ be an action of $G$ on $X$. Let $\beta:G\to Aut(Y)$
be a free action (eg. Bernoulli shift). Define
$\alpha\otimes\beta:G\to Aut(X\times Y)$ by $g(x,y)=(gx,gy)$. With
this definition $\alpha\otimes\beta$ is a free action of $G$, so
it is sofic. We can embed $L^\infty(X\times
Y)\rtimes_{\alpha\otimes\beta}G$ in some $\Pi_{k\to\omega}M_{n_k}$
satisfying the requirements of sofic action. The space
$L^\infty(X)$ can be embedded in $L^\infty(X\times Y)$ by
$id\otimes 1$. This embedding can be extended to an embedding of
$L^\infty(X)\rtimes_\alpha G$ in $L^\infty(X\times
Y)\rtimes_{\alpha\otimes\beta}G$. This will prove $\alpha$ is
sofic.
\end{proof}

\begin{de}
Denote by $\mathcal{S}$ the class of groups for which every action
is sofic.
\end{de}

While we can not prove that every sofic group is in $\mathcal{S}$,
we will provide some examples. First goal is to deal with amenable
groups.

\begin{p}\label{zactions}
Each action of the integers admits a sofic embedding.
\end{p}
\begin{proof}
Let $\alpha:L^\infty(X)\to L^\infty(X)$ the automorphism that
generates the action. Choose $\Theta:L^\infty(X)\to
\Pi_{k\to\omega}D_{n_k}$ an embedding. Apply proposition
\ref{linkabelian} to $\Theta$ and $\Theta\circ\alpha$ to get a
unitary $u\in\Pi_{k\to\omega}P_{n_k}$ such that
$Adu\circ\Theta=\Theta\circ\alpha$.

As powers of permutation matrices are still permutation matrices,
we have $u^m\in\Pi_{k\to\omega}P_{n_k}$. Also
$u^m\Theta(f)(u^m)^*=\Theta(\alpha^m(f))$ for any $f\in
L^\infty(X)$. Now we have an embedding $\Theta$ of the algebraic
crossed product $L^\infty(X)\rtimes_\alpha\zz$. In order to have
an embedding of the crossed product we need the relation
$Tr(u^m)=0$ for any $m\in\zz^*$.

Let $\Lambda$ be an embedding of  $\zz$ in some
$\Pi_{k\to\omega}P_{r_k}$ using only elements of trace $0$. Define
the embedding $\Theta\otimes\Lambda$ of
$L^\infty(X)\rtimes_\alpha\zz$ in $\Pi_{k\to\omega}M_{n_k\cdot
r_k}$ by:
\begin{align*}
\Theta\otimes\Lambda(T)=&\Theta(T)\otimes 1&\mbox{ for }T\in
L^\infty(X) \\
\Theta\otimes\Lambda(u_g)=&\Theta(u_g)\otimes\Lambda(u_g)&\mbox{
for }g\in\zz.
\end{align*}
This embedding $\Theta\otimes\Lambda$ of the algebraic crossed
product respects the trace of the von Neumann crossed product.
Using the unique feature of the type II case the closure of its
imagine will be the crossed product.
\end{proof}

\begin{p}
Amenable groups are in $\mathcal{S}$.
\end{p}
\begin{proof}
Let $G$ be an amenable group and let $\alpha:G\to Aut(X,\mu)$ be a
free action. Then $E_\alpha$ is amenable. By \cite{CFW} $E_\alpha$
is generated by an action $\beta$ of $\zz$.  By previous
proposition $beta$ is sofic. Because almost all equivalence
classes of $E_\alpha$ are non-finite, $\beta$ is free. Using
proposition \ref{orbit equivalent actions} we deduce $\alpha$ is
sofic. Combined with theorem \ref{freesoficactions}, we get
$G\in\mathcal{S}$.
\end{proof}

Next proposition will enlarge the class of groups for which such
results hold.

\begin{te}\label{amalgamatedactionproduct}
Let $\alpha_1$ and $\alpha_2$ be two sofic actions of $G_1$ and
$G_2$ on the same space $X$. Consider $H$, a common amenable
subgroup of $G_1$ and $G_2$. Assume $\alpha_1$ and $\alpha_2$
coincide on $H$, and this action of $H$ is free. Then the action
$\alpha_1*_H\alpha_2$ of $G_1*_HG_2$ is sofic.
\end{te}
\begin{proof}
Using \ref{differentdimensions} we can construct sofic embeddings
of the two crossed products in the same ultraproduct. So let
$\Theta_i:L^\infty(X)\rtimes G_i\to\Pi_{k\to\omega}M_{n_k}$,
$i=1,2$. By \ref{linkamenable} we can assume $\Theta_1=\Theta_2$
on $L^\infty(X)\rtimes H$ (using the freeness of this action). Now
we can construct a representation $\Theta$ of the algebraic
crossed product $L^\infty(X)\rtimes(G_1*_HG_2)$ on
$\Pi_{k\to\omega}M_{n_k}$. In order to embed the von Neumann
crossed product the trace of each nontrivial $u_g$, $g\in
G_1*_HG_2$ must be equal to $0$. Because $L(G)\subset
L^\infty(X)\rtimes G$ we know that $G_1$ and $G_2$ are sofic. Then
$G_1*_HG_2$ is sofic (see \ref{amalgamatedgroupproduct}). There
exist an embedding $\Lambda$ of $G_1*_HG_2$ in some
$\Pi_{k\to\omega}P_{r_k}$ using only elements of trace $0$. Define
the embedding $\Theta\otimes\Lambda$ of
$L^\infty(X)\rtimes_{\alpha_1*_H\alpha_2}G_1*_HG_2$ like in
\ref{zactions}.
\end{proof}

Adapting the same methods we can prove this result for a countable
family of actions.

\begin{p}\label{countable product actions}
Let $\{\alpha_i\}_{i\in\nz}$  be a family of sofic actions of
$\{G_i\}_{i\in\nz}$ on the same space. Assume $H$ is an amenable
common subgroup of $G_i$ and the actions $\alpha_i$ coincide on
$H$. Then $*_H\alpha_i$ is sofic.
\end{p}

\begin{cor}\label{actions of finfty}
Each action of a free group, including $\ff_\infty$ is sofic.
\end{cor}
\begin{proof}
Corollary of \ref{zactions} and \ref{countable product actions}.
\end{proof}

We now recover with our methods the result of Elek-Lippner that
any treeable equivalence relation is sofic.

\begin{p}\label{treeable}
Every treeable equivalence relation is sofic.
\end{p}
\begin{proof}
Well, treeable is some kind of freeness and freeness in general
goes well with soficity.

Let $E$ be a treeable equivalence relation on $(X,\mu)$. Fix a
treeing of $E$, i.e. a countable set of partial Borel isomorphism
$\{\phi_i\}_{i\in\nz^*}\subset[[E]]$. For each $i$ we have
$\phi_i=a_i\lambda_i$, where $a_i$ is a projection in
$L^\infty(X)$ and $\lambda_i\in[E]$.

Define an action $\alpha$ of $\ff_\infty$ on $X$ such that
$\alpha(\gamma_i)=\lambda_i$ (where $\{\gamma_i\}_i$ are the
generators of $\ff_\infty$). Being an action of $\ff_\infty$,
$\alpha$ is sofic.

The von Neumann subalgebra of
$L^\infty(X)\rtimes_\alpha\ff_\infty$ generated by
$a_iu_{\gamma_i}$ is naturally isomorphic to $M(E)$. Hence every
sofic embedding of $L^\infty(X)\rtimes_\alpha\ff_\infty$ can be
restricted to a sofic embedding of $M(E)\subset
L^\infty(X)\rtimes_\alpha\ff_\infty$.
\end{proof}

We end this section with the following theorem.

\begin{te}
Class $\mathcal{S}$ is closed under amalgamated product over
amenable groups. It is strictly larger than the class of treeable
groups.
\end{te}
\begin{proof}
First part of the theorem is \ref{amalgamatedactionproduct} and
\ref{freesoficactions}. By \ref{treeable} (and again
\ref{freesoficactions}) every treeable group is in $\mathcal{S}$.

Consider now the group $G=\zz*_{(2,3)\zz}\zz$. It is not treeable
but $G\in\mathcal{S}$. This example is from \cite{Ga}.

Relation $G\in\mathcal{S}$ is just \ref{zactions} and
\ref{amalgamatedactionproduct}. By general theory of Gaboriau, the
cost of $G$ is $1+1-1=1$. If amalgamation is done with good
morphism (multiplication by 2 and 3) then $G$ is not amenable.
This implies $G$ is not treeable.
\end{proof}

\section{Sofic equivalence relations}\label{Ser}

Now we shall present from \cite{El-Li} the original definition of
Elek and Lippner of soficity for actions and equivalence
relations.

\begin{de}
We call a \emph{basic sequence of projections} for $L^\infty(X)$ a
collection $\{e_{i,m}\}_{1\leq i\leq2^m,m\geq0}\subset
L^\infty(X)$ with following properties:
\begin{enumerate}
\item $\overline{span}^w\{e_{i,m}\}_{i,m}=L^\infty(X)$; \item
$\mu(e_{i,m})=2^{-m},1\leq i\leq 2^m,m\geq0$;\item
$e_{2i-1,m}+e_{2i,m}=e_{i,m-1},m\geq1$.
\end{enumerate}
\end{de}

Let $\ff_\infty=<\gamma_1,\gamma_2,\ldots>$. For any $r\in\nz$
denote by $W_r$ the subset of reduced words of length at most $r$
containing only the first $r$ generators and their inverses. We
have $W_0\subset W_1\subset\ldots$ and $\ff_\infty=\cup_{r\geq 0}
W_r$.

Let $\alpha:\ff_\infty\curvearrowright X$ a Borel action and fix
$\{e_{i,r}\}_{1\leq i\leq 2^r}$ a basic sequence of projections
for $L^\infty(X)$. The following definition will allow us to keep
track of the position of a point $x\in X$ relative to sets
$\{e_{i,r}\}_{1\leq i\leq 2^r}$ under the action of
$W_r\subset\ff_\infty$.

\begin{de}
Let $r\in\nz$. A \emph{r-labeled,\ r-neighborhood} is a finite
oriented multi-graph containing:
\begin{enumerate}
\item a root vertex such that any vertex is connected to the root
by a path of lenght at most $r$; \item ever vertex has a label
from the set $\{1,\ldots,2^r\}$; \item out-edges of every vertex
have different colors from the set
$\{\gamma_1,\gamma_1^{-1},\ldots,\gamma_r,\gamma_r^{-1}\}$; \item
if edge $xy$ is colored with $\gamma_i$ then $yx$ is colored by
$\gamma_i^{-1}$.
\end{enumerate}
Isomorphism classes of such objects form a finite set that we
shall denote by $U^{r,r}$.
\end{de}

For $G\in U^{r,r}$, denote by $R_G$ the root vertex in $G$. For
$\gamma\in W_r$ let $\gamma R_G$ be the vertex in $G$ obtained by
starting from $R_G$ and following the path given by $\gamma$ (if
such a path exists). Finally, let $l(\gamma R_G)$ be the label of
the vertex $\gamma R_G$ in the set $\{1,2,\ldots,2^r\}$.

Let $X$ be a space together with a basic sequence of projections.
For an action $\alpha:\ff_\infty\curvearrowright X$ and $x\in X$
we can define $B_r^r(x)\in U^{r,r}$ by taking the imagines of $x$
under $W_r$ and their labels with respect to $\{e_{i,r}\}_{1\leq
i\leq 2^r}$. For $G\in U^{r,r}$ let $T(\alpha,G)=\{x\in
X:B_r^r(x)\equiv G\}$. Define also $p_G(\alpha)=\mu(T(\alpha,G))$.

If $\alpha$ is an action on a finite space $Y$ (having the
normalized cardinal measure) we have the same definitions provided
that we still have some subsets $\{e_{i,r}\}_{1\leq i\leq
2^r,r\geq 0}$ of $Y$ satisfying the same summation relations. This
are needed to give labels to our vertices. Finite spaces with this
kind of partitions are called $X$-sets. We are now ready to give
the definition.

\begin{de}
An action $\alpha$ of $\ff_\infty$ is called \emph{sofic} (in
Elek-Lippner sense) if there exists a sequence of actions
$\alpha_k$ of $\ff_\infty$ on $X$-sets such that for any $r\geq
1$, for any $G\in U^{r,r}$ we have
$lim_{k\to\infty}p_G(\alpha_k)=p_G(\alpha)$.
\end{de}

\begin{de}
An equivalence relation is called \emph{sofic} if it is generated
by a sofic action of $\ff_\infty$ (all this in Elek-Lippner
sense).
\end{de}

For actions of $\ff_\infty$ the two notions of soficity are
different. With our definition every action of $\ff_\infty$ is
sofic (see \ref{actions of finfty}). Instead for equivalence
relations the two notions are the same. This is what we shall
prove now.

\begin{p}
Let $E\subset X^2$ be a sofic equivalence relation in sense of
Elek-Lippner. Then $E$ is also sofic ($M(E)$ admits a sofic
embedding in some $\Pi_{k\to\omega}M_{n_k}$).
\end{p}
\begin{proof}
Let $\alpha:\ff_\infty\curvearrowright(X,\mu)$ be a sofic action
in the sense of Elek-Lippner such that $E=E_\alpha$. Let
$\alpha_k$ a sequence of actions on X-sets $Y_k$ such that
$lim_{k\to\infty}p_G(\alpha_k)=p_G(\alpha)$. Finally let $n_k$ be
the cardinal of $Y_k$. We shall embed $M(E)$ in
$\Pi_{k\to\omega}M_{n_k}$ in a sofic way. For this we need:
\begin{enumerate}
\item an embedding $L^\infty(X)\subset\Pi_{k\to\omega}D_{n_k}$;
\item a representation $\Theta$ of $\ff_\infty$ on
$\Pi_{k\to\omega}P_{n_k}$; \item the formula
$Tr(f\Theta(\gamma))=\int_{X_\gamma}fd\mu$ for every $f\in
L^\infty(X)$ and $\gamma\in\ff_\infty$, where $X_{\gamma}=\{x\in
X:\gamma x=x\}$.
\end{enumerate}
By hypothesis $Y_k$ are $X$-sets, so they come together with
projections $\{e_{i,r}^k\}_{i,r}$. Construct
$e_{i,r}=\Pi_{k\to\omega}e_{i,r}^k$. We claim that
$\{e_{i,r}\}_{1\leq i\leq 2^r,r\geq 0}$ form a basic sequence of
projections for the algebra they generate. Relations
$e_{i,r}=e_{2i-1,r+1}+e_{2i,r+1}$ are automatic, we just need to
prove that $Tr(e_{i,r})=2^{-r}$.

Let $\{f_{i,r}\}_{1\leq i\leq 2^r,r\geq 0}\subset L^\infty(X)$ the
basic sequence of projections used in the construction of numbers
$p_G(\alpha)$. Fix $i$ and $r$. Let $U^{r,r}_i=\{G\in
U^{r,r}:l(R_G)=i\}$, i.e. graphs such that the root has label $i$.
Then $T(\alpha,G)\subset f_{i,r}$ for each $G\in U^{r,r}_i$.
Moreover: $f_{i,r}=\sqcup_{G\in U^{r,r}_i}T(\alpha,G)$. In the
same way we have $e_{i,r}^k=\sqcup_{G\in U^{r,r}_i}T(\alpha_k,G)$.
Because $\lim_{k\to\infty}p_G(\alpha_k)=p_G(\alpha)$ we have
$Tr(e_{i,r})=\lim_{k\to\omega}Tr(e_{i,r}^k)=Tr(f_{i,r})=2^{-r}$.

By identifying $e_{i,r}$ with $f_{i,r}$ we get an embedding of
$L^\infty(X)$. Now we construct the representation $\Theta$ of
$\ff_\infty$ in $\Pi_{k\to\omega}P_{n_k}$. We identified set $Y_k$
with diagonal $D_{n_k}$ and we have actions $\alpha_k$ of
$\ff_\infty$ that are defined on $Y_k$. This will construct a
representation. We need to make sure $\Theta$ acts the same way as
$\alpha$.

Let $\gamma$ be one of the generators of $\ff_\infty$. Fix $i,j$
and $r$. Let $U^{r,r}_{i,\gamma j}=\{G\in
U^{r,r}:l(R_G)=i,l(\gamma R_G)=j)\}$, the set of graphs such that
the root has label $i$ and the vertex connected with the root by
the $\gamma$ edge has label $j$ (the existence of such an edge is
a requirement we ask now for $G$). It is easy to see that
$f_{i,r}\cap\alpha(\gamma^{-1})(f_{j,r})=\sqcup_{G\in
U^{r,r}_{i,\gamma j}}T(\alpha,G)$. In the same way
$e_{i,r}^k\cap\alpha_k(\gamma^{-1})(e_{j,r}^k)=\sqcup_{G\in
U^{r,r}_{i,\gamma j}}T(\alpha_k,G)$. Using the hypothesis we get
$Tr(e_{i,r}\cdot\Theta(\gamma^{-1})(e_{j,r}))=\mu(f_{i,r}\cap\alpha(\gamma^{-1})(f_{j,r}))$.
This is enough to deduce that the action that $\Theta$ induce on
our embedding of $L^\infty(X)$ is equal to $\alpha$.

For the third requirement let now $\gamma\in\ff_\infty$ an
arbitrary element. It is of course sufficient to assume that $f$
is one of the projections $e_{i,r}$. We need to prove that
$Tr(e_{i,r}\Theta(\gamma))=\mu(X_\gamma\cap e_{i,r})$. Lets say
that in our construction we have
$\Theta(\gamma)=\Pi_{k\to\omega}\gamma_k$. Then
$Tr(e_{i,r}\Theta(\gamma))=lim_{k\to\infty}Tr(e_{i,r}^k\gamma_k)$.
Let $U^{r,r}_{i,\gamma}=\{G\in U^{r,r}:l(R_G)=i,\gamma R_G=R_G\}$,
i.e. the set of $G\in U^{r,r}$ such that the root has label $i$
and the path in $G$ described by $\gamma$, starting from the root,
returns to the root. Then $X_\gamma\cap e_{i,r}=\sqcup_{G\in
U^{r,r}_{i,\gamma}}T(\alpha,G)$. A similar formula with fixed
points of $\gamma_k$ and $e_{i,r}^k$ takes place. By
$lim_{k\to\infty}p_G(\alpha_k)=p_G(\alpha)$ we get
$Tr(e_{i,r}\Theta(\gamma))=\mu(X_\gamma\cap e_{i,r})$ and we are
done.
\end{proof}

\begin{p}
Let $E$ be a sofic equivalence relation ($M(E)$ embeds in some
$\Pi_{k\to\omega}M_{n_k}$). Then $E$ is also sofic in the sense of
Elek-Lippner.
\end{p}
\begin{proof}
By \ref{soficembedding} we have a sofic embedding
$M(E)\subset\Pi_{k\to\omega}M_{n_k}$ such that
$L^\infty(X)=A\subset\Pi_{k\to\omega}D_{n_k}$ and
$u_\theta\subset\Pi_{k\to\omega}P_{n_k}$ for any $\theta\in[E]$.

We shall denote by $d$ the normalized Hamming distance on
$P_{n_k}$. In general $\gamma,\delta$ will denote elements in
$\ff_\infty$ and $\gamma_i,\delta_i$ will denote generators of
$\ff_\infty$. Let $\alpha:\ff_\infty\curvearrowright(X,\mu)$ an
action that generates the equivalence relation $E$ on $X$. For any
element $\gamma\in\ff_\infty$, $\alpha(\gamma)$ induce an element
$u_\gamma\in\nr(A)$ and $u_\gamma u_\delta=u_{\gamma\delta}$. We
shall write $u_\gamma=\Pi_{k\to\omega}
u_\gamma^k\in\Pi_{k\to\omega}P_{n_k}$.

Let $Y_k$ be a set with $n_k$ elements and identify algebra
$D_{n_k}$ with $L^\infty(Y_k)$. For any generator $\gamma_i$ of
$\ff_\infty$, $u_{\gamma_i}^k\in P_{n_k}$ induce an automorphism
of $Y_k$. Denote it by $\alpha_k(\gamma_i)$ and extend $\alpha_k$
by multiplicity to an action of $\ff_\infty$.

Let $\{e_{i,m}\}_{1\leq i\leq2^m,m\geq0}\subset L^\infty(X)$ a
basic sequence of projections. Use it in order to construct sets
$T(\alpha,G)$. Write $e_{i,m}=\Pi_{k\to\omega}e_{i,m}^k$ such that
$\{e_{i,m}^k\}_{i,m}$ respect the same summation relations. Now
elements in $Y_k$ are labelled by projections
$\{e_{i,m}^k\}_{i,m}$ so we have ingredients for constructing sets
$T(\alpha_k,G)$.

We need to show that out of this actions we can find a subsequence
satisfying the definition of soficity, namely that
$lim_{k\to\infty}\mu_{n_k}(T(\alpha_k,G))=\mu(T(\alpha,G))$ for
any $r\in\nz$ and any $G\in U^{r,r}$ (denote by $\mu_{n_k}$ the
normalized cardinal measure on a set with $n_k$ elements). The
subsequence is just to get rid of the ultrafilter and obtain
classical limit for the countable set of objects that we are
working with.

Fix $r\in\nz$ and $\vp>0$. Let us see that it is enough to find
$k\in\nz$ such that $|\mu_{n_k}(T(\alpha_k,G))-\mu(T(\alpha,G))|
<\vp$ for any $G\in U^{r,r}$. The phenomena here is that, if we
fix a $G\in U^{r,r}$, when we pass from step $r$ to $r+1$ we have
$T(\alpha,G)=\cup_{G'\in U^{r+1,r+1};G<G'}T(\alpha,G')$ (relation
$G<G'$ is defined in an obvious way). So $\mu(T(\alpha_{k'},G))$
is a sum of other $\mu(T(\alpha_{k'},G'))$, but a finite sum. When
we choose our sequence $\{\vp_r\}$ we have to make sure that it
compensates this growth.

Sets $\{T(\alpha,G):G\in U^{r,r}\}$ form a partition of $X$. Let
$T(\alpha,G)=\Pi_{k\to\omega}T(\alpha,G)^k$ such that
$\{T(\alpha,G)^k:G\in U^{r,r}\}$ is a partition of $Y_k$. We also
have in $D_{n_k}$ projections $T(\alpha_k,G)$. We know that
$\mu_{n_k}(T(\alpha,G)^k)\to_k\mu(T(\alpha,G))$ and we want to
show that $\mu_{n_k}(T(\alpha_k,G))\to_k\mu(T(\alpha,G))$ .

Now fix $G\in U^{r,r}$. We need to understand equations that
describe points in $T(\alpha,G)$. Remember that $R_G$ is the root
vertex in $G$; for $\gamma\in W_r$, $\gamma R_G$ is the vertex in
$G$ obtained by starting from $R_G$ and following the path given
by $\gamma$ (if such a path exists). Finally, $l(\gamma R_G)$ is
the label of the vertex $\gamma R_G$ in the set
$\{1,2,\ldots,2^r\}$. We can now state our characterization of
$T(\alpha,G)$.

A point $x\in X$ is an element of the set $T(\alpha,G)$ iff:
\begin{enumerate}
\item $\alpha(\gamma)(x)\in e_{l(\gamma R_G),r}$ for any
$\gamma\in W_r$ for which $\gamma R_G$ exists; \item
$\alpha(\gamma)(x)=\alpha(\delta)(x)$ $\forall \gamma,\delta\in
W_r,\ \gamma R_G=\delta R_G$; \item
$\alpha(\gamma)(x)\neq\alpha(\delta)(x)$ $\forall \gamma,\delta\in
W_r,\ \gamma R_G\neq\delta R_G$.
\end{enumerate}

First condition gives the coloring of vertices. The other two give
the structure of the graph $G$. Let $\vp_1>0$ such that
$2|U^{r,r}|(|W_r|+2|W_r|^2)\vp_1<\vp$. We want to find $k\in\nz$
such that for any $G\in U^{r,r}$ we have:
\begin{align}
 &\mu_{n_k}\left( \alpha_k(\gamma)(T(\alpha,G)^k)\setminus
e_{l(\gamma R_G),r}^k\right)<\vp_1 &\forall \gamma\in W_r;\label{e1}\\
 &\mu_{n_k}\left(T(\alpha,G)^k\setminus \{y\in Y_k:
\alpha_k(\gamma)(y)= \alpha_k(\delta)(y)\}\right) <\vp_1
&\forall \gamma,\delta\in W_r,\ \gamma R_G=\delta R_G\label{e2}\\
&\mu_{n_k}\left(T(\alpha,G)^k\setminus \{y\in Y_k:
\alpha_k(\gamma)(y)\neq\alpha_k(\delta)(y)\}\right) <\vp_1
&\forall \gamma,\delta\in W_r,\ \gamma R_G\neq\delta R_G\label{e4}\\
&|\mu_{n_k}(T(\alpha,G)^k)-\mu(T(\alpha,G))| <\vp/2\label{e3}.
\end{align}

First we shall prove that this four conditions are enough to
guarantee $|\mu_{n_k}(T(\alpha_k,G))-\mu(T(\alpha,G))| <\vp$ for
every $G\in U^{r,r}$. Using (\ref{e3}) we just need to prove
$|\mu_{n_k}(T(\alpha_k,G))-\mu_{n_k}(T(\alpha,G)^k)| <\vp/2$.

Take $x\in (T(\alpha,G)^k\setminus T(\alpha_k,G))$ for some $G\in
U^{r,r}$. Following our characterization of $T(\alpha,G)$, we
have:
\begin{enumerate}
\item $\exists\gamma\in W_r$ such that $\alpha_k(\gamma)(x)$ does
not have the right label, namely $l(\gamma R_G)$; \item or
$\exists\gamma,\delta\in W_r$ such that $\gamma R_G=\delta R_G$
and $\alpha_k(\gamma)(x)\neq \alpha_k(\delta)(x)$; \item or
$\exists\gamma,\delta\in W_r$ such that $\gamma R_G\neq\delta R_G$
and $\alpha_k(\gamma)(x)= \alpha_k(\delta)(x)$.
\end{enumerate}
Using (\ref{e1}), (\ref{e2}) and (\ref{e4}) we get:
\begin{align*}
\mu_{n_k}(T(\alpha,G)^k\setminus
T(\alpha_k,G))<|W_r|\vp_1+2|W_r|^2\vp_1.
\end{align*}
Because both $\{T(\alpha_k,G)\}_G$ and $\{T(\alpha,G)^k)\}_G$ are
partitions of $Y_k$ and the above formula holds for any $G\in
U^{r,r}$ we have: \[|\mu_{n_k}(T(\alpha_k,G))-
\mu_{n_k}(T(\alpha,G)^k)|<|U^{r,r}|(|W_r|\vp_1+2|W_r|^2\vp_1)<\vp/2.\]

Now back to the choice of $k$. Let $\gamma=\gamma_{i_1}
\gamma_{i_2}\ldots\gamma_{i_s}\in W_r$. We should in fact take
$\gamma=\gamma_{i_1}^{\zeta_1}\gamma_{i_2}^{\zeta_2}\ldots
\gamma_{i_s}^{\zeta_s}$, where $\zeta_j\in\{\pm1\}$. The inverses
will change nothing in our arguments and will only overload our
notations. Due to Feldman-Moore construction we know that
$u_\gamma=u_{\gamma_{i_1}}u_{\gamma_{i_2}}\ldots
u_{\gamma_{i_s}}$. Next, combine
$\alpha(\gamma)(T(\alpha,G))\subset e_{l(\gamma R),r}$ and
$\alpha(\gamma)(T(\alpha,G))=u_\gamma T(\alpha,G)u_\gamma^*$ to
get $u_\gamma T(\alpha,G)u_\gamma^*\subset e_{l(\gamma R),r}$.

Consider now $\gamma,\delta\in W_r$. If $\gamma R_G=\delta R_G$
then $\alpha(\gamma)|_{T(\alpha,G)}=
\alpha(\delta)|_{T(\alpha,G)}$ so
$Tr(T(\alpha,G)u_\delta^*u_\gamma)= \mu(T(\alpha,G))$ (here we
consider $T(\alpha,G)$ to be a projection of
$L^\infty(X)\subset\Pi_{k\to\omega}M_{n_k}$). If $\gamma
R_G\neq\delta R_G$ then $Tr(T(\alpha,G)u_\delta^* u_\gamma)=0$.
Find $k\in\nz$ such that (\ref{e3}) holds and:
\begin{align}
&||u_\gamma^k-u_{\gamma_{i_1}}^k u_{\gamma_{i_2}}^k\ldots
u_{\gamma_{i_s}}^k||_2<\vp_1/4&\forall \gamma=\gamma_{i_1}
\gamma_{i_2}\ldots\gamma_{i_s}\in W_r\label{in3}\\
&\mu_{n_k}\left(u_{\gamma}^k T(\alpha,G)^ku_{\gamma}^{k*}\setminus
e_{l(\gamma R_G),r}^k\right)<\vp_1/2&\forall \gamma\in W_r,
\forall
G\in\ U^{r,r}\label{in4}\\
&\mu_{n_k}(T(\alpha,G)^k)-Tr\left(T(\alpha,G)^ku_\delta^{k*}u_\gamma^k\right)<
\vp_1/2 &\forall G\in U^{r,r}\forall\gamma,\delta\in W_r,\gamma
R_G= \delta R_G\label{in5}\\
&Tr\left(T(\alpha,G)^ku_\delta^{k*}u_\gamma^k\right)< \vp_1/2
&\forall G\in U^{r,r}\forall\gamma,\delta\in W_r,\gamma R_G\neq
\delta R_G\label{in8}
\end{align}

By definition $\alpha_k(\gamma)=\alpha_k(\gamma_{i_1})
\alpha_k(\gamma_{i_2})\ldots\alpha_k(\gamma_{i_r})$ and
$\alpha_k(\gamma_{i_j})(P)=u_{\gamma_{i_j}}^kPu_{\gamma_{i_j}}^{k*}$
for any projection $P\in D_{n_k}$. Then:
\begin{align}
\alpha_k(\gamma)(T(\alpha,G)^k)=\left(u_{\gamma_{i_1}}^k
u_{\gamma_{i_2}}^k\ldots u_{\gamma_{i_r}}^k\right)
T(\alpha,G)^k\left(u_{\gamma_{i_1}}^k u_{\gamma_{i_2}}^k\ldots
u_{\gamma_{i_r}}^k\right)^*\label{in6}. \end{align} Both
$u_\gamma^k$ and $u_{\gamma_{i_1}}^k u_{\gamma_{i_2}}^k\ldots
u_{\gamma_{i_s}}^k$ are elements in $P_{n_k}$ so by (\ref{in3}):
\begin{align*}
d(u_\gamma^k,u_{\gamma_{i_1}}^k u_{\gamma_{i_2}}^k\ldots
u_{\gamma_{i_s}}^k)<\vp_1/4.
\end{align*}
Combined with (\ref{in6}), we have
$\mu_{n_k}\left(\alpha_k(\gamma) (T(\alpha,G)^k)\setminus
u_\gamma^k T(\alpha,G)^ku_\gamma^{k*}\right)<\vp_1/4$. Use
(\ref{in4}) to get $\mu_{n_k}\left(\alpha_k(\gamma)
(T(\alpha,G)^k)\setminus e_{l(\gamma
R),r}^k\right)<\vp_1/4+\vp_1/2<\vp_1$, so we have (\ref{e1}).

Let now $\gamma=\gamma_{i_1} \gamma_{i_2}\ldots\gamma_{i_s}$ and
$\delta=\delta_{j_1} \delta_{j_2}\ldots\delta_{j_t}$ such that
$\gamma R_G=\delta R_G$. Use (\ref{in3}) both for $\gamma$ and
$\delta$ to get:
\begin{align*}
||u_\delta^{k*}u_\gamma^k-(u_{\delta_{j_1}}^k
u_{\delta_{j_2}}^k\ldots u_{\delta_{j_t}}^k)^*(u_{\gamma_{i_1}}^k
u_{\gamma_{i_2}}^k\ldots u_{\gamma_{i_r}}^k)||_2<\vp_1
\end{align*}
As before, $u_\delta^{k*}u_\gamma^k$ and $(u_{\delta_{j_1}}^k
\ldots u_{\delta_{j_t}}^k)^*(u_{\gamma_{i_1}}^k\ldots
u_{\gamma_{i_r}}^k) $ are elements in $P_{n_k}$, so:
\begin{align*}
d\left(u_\delta^{k*}u_\gamma^k,(u_{\delta_{j_1}}^k
u_{\delta_{j_2}}^k\ldots u_{\delta_{j_t}}^k)^*(u_{\gamma_{i_1}}^k
u_{\gamma_{i_2}}^k\ldots u_{\gamma_{i_r}}^k)\right)<\vp_1/2,
\end{align*}
Restricting this inequality just to fixed points in set
$T(\alpha,G)^k$, we get:
\begin{align}
|Tr(T(\alpha,G)^ku_\delta^{k*}u_\gamma^k)-
Tr\left(T(\alpha,G)^k(u_{\delta_{j_1}}^k u_{\delta_{j_2}}^k\ldots
u_{\delta_{j_t}}^k)^*(u_{\gamma_{i_1}}^k u_{\gamma_{i_2}}^k\ldots
u_{\gamma_{i_r}}^k)\right)|<\vp_1/2\label{in7}.
\end{align}
Apply (\ref{in6}) for $\gamma$ and $\delta$ to get:
\begin{align*}
\mu_{n_k}\left(T(\alpha,G)^k\setminus \{y\in Y_k:
\alpha_k(\gamma)(y)= \alpha_k(\delta)(y)\}\right)=
\mu_{n_k}(T(\alpha,G)^k)-\\ Tr\left(T(\alpha,G)^k
(u_{\delta_{j_1}}^k u_{\delta_{j_2}}^k\ldots
u_{\delta_{j_t}}^k)^*(u_{\gamma_{i_1}}^k u_{\gamma_{i_2}}^k\ldots
u_{\gamma_{i_r}}^k)\right).
\end{align*}
This combined with (\ref{in7}) and(\ref{in5}) yields (\ref{e2}).

Assume now $\gamma R_G\neq\delta R_G$. Then:
\begin{align*}
\mu_{n_k}(T(\alpha,G)^k\setminus &\{y\in Y_k:
\alpha_k(\gamma)(y)\neq \alpha_k(\delta)(y)\})=\\
&Tr\left(T(\alpha,G)^k  (u_{\delta_{j_1}}^k
u_{\delta_{j_2}}^k\ldots u_{\delta_{j_t}}^k)^*(u_{\gamma_{i_1}}^k
u_{\gamma_{i_2}}^k\ldots u_{\gamma_{i_r}}^k)\right).
\end{align*}
Inequalities (\ref{in7}) and (\ref{in8}) will imply (\ref{e4}).
\end{proof}

\section*{Acknowledgements}

It is my great pleasure to thanks Florin R\u adulescu for many
discussions and ideas. Special thanks to my friend and colleague
Valerio Capraro for introducing me to the subject of \emph{Connes'
embedding problem} and for his study and work on the subject from
which I benefited. I am very grateful to Stefaan Vaes for numerous
remarks and corrections on previous versions of the paper and also
for considerably easier proofs for some results including lemma
\ref{linkabelian} and lemma \ref{zactions}. Parts of this article
were written during my stay in Leuven in Spring 2010. Also, I want
to thank Damien Gaboriau and Ken Dykema for useful remarks.

L. P\u AUNESCU, \emph{UNIVERSIT\`{A} DI ROMA TOR VERGATA and
INSTITUTE of MATHEMATICS "S. Stoilow" of the ROMANIAN ACADEMY}
email: liviu.paunescu@imar.ro

\end{document}